 \documentclass[draft]{article}

\usepackage{amsmath,amsfonts,amsthm,amssymb,amscd,cancel,color,mathtools}
\usepackage{enumitem}
\usepackage{verbatim}
\usepackage{ulem,color}

\renewcommand{\Re}{\mathop{\rm Re}\nolimits}
\renewcommand{\Im}{\mathop{\rm Im}\nolimits}

\def\uno{{\kern+.3em {\rm 1} \kern -.22em {\rm l}}}

\theoremstyle{plain} \newtheorem{theorem}{Theorem}[section]
\newtheorem{lemma}[theorem]{Lemma}
\newtheorem{proposition}[theorem]{Proposition}
 \theoremstyle{definition}
\newtheorem{definition}[theorem]{Definition} \theoremstyle{remark}
\newtheorem{remark}[theorem]{Remark}

\newtheorem{claim}[theorem]{Claim}

\newcommand{\R}{{\mathbb R}}

\def\im{{\rm i}}

\newcommand{\C}{\mathbb{C}}

\def\({\left(}
\def\){\right)}
\def\<{\left\langle}
\def\>{\right\rangle}

\def\[{\left [}
\def\]{\right ]}

\numberwithin{equation}{section}

\begin{document}

\title{On stability  of small solitons of the 1--D NLS with a  trapping delta potential}

\author {Scipio Cuccagna, Masaya Maeda}

\maketitle

\begin{abstract} We consider a  Nonlinear Schr\"odinger Equation with a very general non linear term and  with a trapping $\delta $ potential on the line. We then discuss    the asymptotic behavior of all its small solutions, generalizing a recent result by Masaki \textit{et al.} \cite{MMS1}. We give also a result of dispersion in the case of defocusing   equations
with a non--trapping delta potential.

\end{abstract}

\section{Introduction}
\label{sec:introduction}

In this paper we consider the   Nonlinear Schr\"odinger Equation (NLS)
\begin{equation}\label{1}
\im \dot u  =    H_1 u  + g(|u|^2)  u ,\quad (t,x)\in \R\times \R ,\text{with $u(0)=u_0 \in H^1 (\R , \C )$},
\end{equation}
with the Schr\" odinger operator (here $\delta (x) $ is the Dirac $\delta$ centered in 0)
\begin{align}\label{schop}
 & H_q= -\partial ^2_x  -q\delta (x)    \text{  for }  q\in \R \backslash \{ 0 \}
\end{align}
defined by  $H_q:= -\partial ^2_x   $   with domain
\begin{align}\label{domain}
 &   D(H_q) =  \{u\in H^1(\R , \C )\cap H^2(\R\setminus\{0\}, \C )\ |\ \partial_xu(0^+)-\partial _xu(0^-)=-qu(0)\} .
\end{align}
For the nonlinearity, we assume $g\in C([0,\infty),\R)\cap C^3((0,\infty),\R)$ and that there exist  $p>0$  and $C>0$ s.t. for $k=0,1,2,3$ we have
\begin{align}\label{nonlinearity}
 |g^{(k)}(s)| \leq C |s|^{p-k}  \text{    for all $s\in(0,1]$.}
\end{align}
In particular, we have $g(0)=0$ and  the primitive $G$ of $g$ defined by
\begin{equation}\label{eq:G}
    G'(s) = g (s) \text{  and }G (0) =0.
\end{equation}
 satisfies $|G(s)|\lesssim |s|^{p+1}$ for all $s\in(0,1)$.
\begin{remark}
A typical example we have in mind is $g(s)=\lambda s^p$ with $p>0 $ and $\lambda\in \{\pm 1\}$.
In this case, our NLS   can be written taking $q=1$ in the form
\begin{equation}\label{1prime}
\im \dot u  =    H_q u  + \lambda |u|^{2p}  u ,
\end{equation}
which was considered by Masaki \textit{et  al.} \cite{MMS1} for the case $p\geq 2$. They also
considered the cubic NLS for the cubic NLS,   $p=1$, with dispersive potential $q<0$
in \cite{MMS2},     where they proved dispersion, that is $\| u(t) \| _{L^\infty (\R )}\lesssim t^{-\frac{1}{2}}$ as $t\to +\infty$,  for appropriate very small solutions.
\end{remark}

We recall, see  \cite{MMS1}, that the operator in \eqref {schop}  for $q>0$ satisfies
 \begin{align}\label{schopSPec}
\text{$\sigma _d(H_q)=\{    - q^2/4\}$
 with $\ker\(  H_q+ q^2/4 \) =\text{Sp} (  \varphi _q   )$ where $ \varphi _q :=  \sqrt{q/2}  e^{-\frac{q}{2} |x|}$ ,}
\end{align}
with $\text{Sp} (  \varphi _q   ):=\C \varphi _q$.
Furthermore the point 0 is neither an eigenvalue nor a resonance for $H_q$, that is to say,
the only $u_0\in L^2 (\R ) \cup L^\infty (\R ) $ s.t. $ H_q u_0=0 $ is $u_0=0$.

\noindent We also have a spectral (orthogonal) decomposition
\begin{align}\label{schopSPec1}
L^2(\R ) = \text{Sp} (  \varphi _q   )\oplus L^2_c(H_q)
\end{align}
with $L^2_c(H_q)$ the  continuous spectrum component  associated
to $H_q$. We will consider the case $q=1$ and denote
\begin{align*}
L^2_c:=L^2_c(H_1)\text{ and }\varphi:=\varphi_1.
\end{align*}
We will denote by $P_c$ the projection  onto $L^2_c$.
In particular,
\begin{align*}
P_cu:=u-\(\int_\R u\varphi\,dx\) \varphi =u -\<u,\varphi\>\varphi-\<u,\im \varphi\>\im\varphi,
\end{align*}
where
 \begin{equation}\label{eq:bilf}
    \langle f , g \rangle = \Re \int _{\R  } f(x) \overline{g}(x) dx \text{  for  $f,g:\R \to \C$ }.
 \end{equation}
 We will also use the following notation.
\begin{itemize}
  \item Given a Banach space $X$, $v\in X$ and $\delta>0$ we set
$
D_X(v,\delta):=\{ x\in X\ |\ \|v-x\|_X<\delta\}.
$
\item
For $\gamma\in \R$ we set
\begin{align}
L^2_\gamma&:=\{u\in \mathcal S'(\R;\C)\ |\ \|u\|_{L^2_\gamma}:=\|e^{\gamma|x|}u\|_{L^2}<\infty\},\label{L2g}\\
H^1_\gamma&:=\{u\in \mathcal S'(\R;\C)\ |\ \|u\|_{H^1_\gamma}:=\|e^{\gamma|x|}u\|_{H^1}<\infty\}.\label{H1g}
\end{align}
\item
For $f:\C\to X$ for some Banach space $X$, we set $D_1 f =\partial_{\Re z} f$ and $D_2 f=\partial_{\Im z}f$.
\end{itemize}

\noindent  The eigenvalue of $H_1$ yields by bifurcation   a family of standing waves solutions.
%
%
%

As in \cite{GNT,CM15APDE,MMS1},  we have the following, which we prove in the appendix.
\begin{proposition}[Bound states]\label{prop:bddst} Let $p>0$.
Then there exists $\gamma_0>0$, $a_0>0$ and $C>0$ s.t. there exists a unique $Q\in C^1(D_\C(0,a_0);H^1_{\gamma_0})$ satisfying the gauge property
\begin{align}\label{Q:gauge}
Q[e^{\im \theta}z]=e^{\im \theta}Q[z],
\end{align}
s.t. there exists $E\in C([0,a_0^2),\R)$ s.t.
\begin{align}\label{eq:sp}
H_1 Q[z]+g(|Q[z]|^2)Q[z]=E(|z|^2)Q[z],
\end{align}
and
\begin{align}\label{prop:bddst1}
\|Q[z]-z\varphi\|_{H^1_{\gamma_0}}\leq C |z|^{ 2p   +1},\ \|D_jQ[z]-\im^{j-1}\varphi\|_{H^1_\gamma}\leq C |z|^{ 2p  },\
\left|E(|z|^2)+\frac14\right|\leq C |z|^{ 2p   }.
\end{align}
Moreover,
if $2p>1$  we have
\begin{align}\label{prop:bddst3}
Q[z] \in C^2 \( D_\C ( 0, a _0 ) , H^1_{\gamma_0} \) ,
\end{align}
and
\begin{align}\label{prop:bddst4}
\| D_jD_kQ[z] \|_{H^1_\gamma}\leq C |z|^{2p-1},\quad j,k=1,2.
\end{align}
\end{proposition}

\begin{remark}
In the case of  power type nonlinearities $g(s)=s^p$, there is an explicit formula for $Q[z]$.
See \cite{ozawa,MMS1}.
\end{remark}

Our first result is the following, related to \cite{GNT,mizumachi,MMS1}, see \cite{CM15APDE}  for   more references.
\begin{theorem}\label{thm:small en}  Assume    $p>0$  in \eqref{nonlinearity}. Then there exist  $\epsilon _0 >0$, $\gamma>0$ and $C>0$ such that  for $\epsilon :=\| u (0)\| _{H^1}<\epsilon _0  $ the  solution  $u(t)$ of  \eqref{1} can be written uniquely  for all times as
 \begin{equation}\label{eq:small en1}
\begin{aligned}&    u(t)= Q[z (t)]+\xi (t) \text{ with $\xi (t) \in
P_c H^1$,}
\end{aligned}
\end{equation}
s.t.  we have
 \begin{equation}\label{eq:small en2}
\begin{aligned}&     |z(t)|+\| \xi (t)\|  _{H^1}\le C \epsilon \text{ for all $t\in [0,\infty ) $, }
\end{aligned}
\end{equation}
\begin{equation}\label{eq:small en3}
\begin{aligned}&     \int _0^\infty    \|\xi\|_{H^1_{-\gamma}}^2 dt \le C \epsilon .
\end{aligned}
\end{equation}

\end{theorem}
Theorem \ref{thm:small en}  claims that   solutions with sufficiently small   $H^1$ norm     converge asymptotically to the set  formed by the $Q[z]$. Indeed  formula \eqref{eq:small en3} is stating that, in an averaged sense, $\xi \xrightarrow {t\to \infty }  0$ locally in space.  In  Theorem \ref{thm:small en}  there is no proof of   \textit{selection of ground state}:  we do not prove that  up to a phase, $z(t)$ has a limit as $t\to +\infty$. However, if
we strengthen the hypotheses of the nonlinearity $g(s)$, we obtain also the selection of ground states. This will be our second result. It requires  a more subtle representation of $u(t)$ than the one in \eqref{eq:small en1}, due to Gustafson \textit{et al.} \cite{GNT}.

\begin{definition}\label{def:contsp}
Consider the     $a _0 >0$ in Proposition \ref{prop:bddst}.
\begin{equation}\label{eq:contsp}
\begin{aligned}
\mathcal{H}_c[z]&  :=\left\{\eta\in L^2(\R ) :\     \<\im \,  {\eta},D _{1} Q \>=  \<\im \, {\eta} ,D _{2} Q\>=0 \right\}.
\end{aligned}
\end{equation}
\end{definition}

It is immediate that   $\mathcal{H}_c[0]=L^2_c $. Our second result is the following.
\begin{theorem}\label{thm:small en-}  Let $p> 1/2$  in  in \eqref{nonlinearity}. Then
 there exist  $\epsilon _0 >0$, $\gamma>0$ and $C>0$ such that  for $\epsilon :=\| u (0)\| _{H^1}<\epsilon _0  $ the  solution  $u(t)$ of  \eqref{1} can be written uniquely  for all times as
 \begin{equation}\label{eq:small en1-}
\begin{aligned}&    u(t)= Q[z (t)]+ \eta (t) \text{ with $\eta (t) \in
\mathcal{H}_c[z(t)]$,}
\end{aligned}
\end{equation}
s.t.  we have
 \begin{equation}\label{eq:small en2-}
\begin{aligned}&     |z(t)|+\| \eta  (t)\|  _{H^1}\le C \epsilon \text{ for all $t\in [0,\infty ) $, }
\end{aligned}
\end{equation}
\begin{equation}\label{eq:small en3-}
\begin{aligned}&     \int _0^\infty    \| \eta\|_{H^1_{-\gamma}}^2 dt \le C \epsilon .
\end{aligned}
\end{equation}
and     there exists  a $z  _+\in \C $ such that \begin{equation}\label{eq:small en22} \begin{aligned}&       \lim _{t\to +\infty } z (t) e^{ \im   \int _0^t E[z(s)]  ds}=z _+. \end{aligned} \end{equation}
\end{theorem}

We don't know if the last statement, with the limit \eqref{eq:small en22},  is true  for $p\le 1/2.$
We will prove Theorem \ref{thm:small en}  in Sect. \ref{sec:thm:small en} and  Theorem \ref{thm:small en-}  in Sect. \ref{sec:thm:small en-}.

Equations like  \eqref{1} and its particular case \eqref{1prime} represent an interesting special type of the NLS in 1--D.   Related models,      obtained eliminating the linear $\delta (x)$ potential and   replacing $ g(|u|^2)  u$  with $\sum _{j=1}^n \delta (x-x_j)g(|u|^2)  u$, in some cases have been shown to satisfy   very satisfactory characterizations of the global   time behavior for all their finite energy solutions; see \cite{Kom03}--\cite{KK10b}, which solved the   \textit{Soliton Risolution Conjecture} in these cases.

Returning to equation  \eqref{1}, Goodman \textit{et. al.} \cite{GHW}
 and Holmer  \textit{et al.} \cite{holmer1}--\cite{holmer4}  have shown in the cubic case  interesting patterns involving solitons, usually for finite intervals of time or numerically. Some of these results have been proved for global times   by    Deift and Park \cite{DP}, using the Inverse Scattering Transform.
  Masaki \textit{et  al.} \cite{MMS1}  is a transposition  to \eqref{1prime}  of a result similar to Theorem \ref{thm:small en-}, but for more regular  potentials, by  Mizumachi \cite{mizumachi}. Similarly,  the result in   Masaki \textit{et  al.} \cite{MMS2} we described under \eqref{1prime}   transposes
   to the case of $\delta $ potentials work on dispersion of very small solutions
   for NLS's with a non--trapping and quite regular potential  in \cite{Delort,Germain,Naumkin}.  Even though they are usually   motivated
   by the problem of stability of solitons,  currently it is not so clear how to get from them
  results of the type $\| \xi(t) \| _{L^\infty (\R )}\lesssim t^{-\frac{1}{2}}$ as $t\to +\infty$ for the error term $\xi(t)$ in \eqref{eq:small en1}  or for the $\eta (t)$ in   \eqref{eq:small en1-}.  Such kind of transformation of results around 0 into results around a soliton
  exists in the context of the theory of Integrable Systems, where there are appropriate coordinate changes named   B\"acklund  and    Darboux  transforms, see for instance Deift and Park \cite{DP}.  However   for non integrable perturbations of the cubic NLS
such coordinate changes  represent  an   open question,  c.f.r.    the  discussion    in  Mizumachi and Pelinovsky  \cite{MP}.

The main  motivation for this paper  is then  to show the promise of an alternative
method,  involving positive commutators, which is classical in Quantum Mechanics, see for example Reed and Simon \cite[pp. 157--163]{RS4} and \cite{amrein,Derezinzki}.  In the nonlinear setting
the method is also classical and  has been extensively used to prove   dispersion, like for example Morawetz estimates, see \cite{caz},  or in the analysis of blow up, see for example \cite{ogawa,MR1}. The method
represents the tool of choice to  complete the last step of the proof in the theory of stabilization developed by Kenig and Merle  \cite{KM06Invent} (a possible alternative, the \textit{energy channel method} of Duyckaerts \textit{et al.} \cite{Duyckaerts}, has not been adapted yet to NLS's).
In this paper we are inspired  by the stability of various patterns studied for wave like equations by Kowalczyk et. al  \cite{KMM2,KMM3}. The main point here,  is that this method
can be applied rather simply in the proof of Theorems \ref{thm:small en}  and  \ref{thm:small en-}.

 \noindent In the nonlinear setting, an insidious problem arises from the  fact that the commutators often  have   some negative
eigenvalues.
 An important part of the proofs in papers such as      \cite{KMM2}  consists in showing that analogues of the $\xi $ in \eqref{eq:small en1} or  of the $\eta$ in \eqref{eq:small en1-}, live where the commutator is positive.  If in \eqref{1} we replace
the $ \delta$--potential with a more regular one, proving such positivity, or circumventing possible negativity, appears to be mostly an open problem.  See \cite{Tao} for related problems.
  In the case of a $ \delta$--potential in 1--D, we show that
this is an easy problem (see also Banica and Visciglia \cite{banica}, Ikeda and Inui \cite{ikeda} and Richard \cite{richard}).
This allows us to  cover with a rather simple proof cases outside the reach of the theory in    Masaki \textit{et  al.} \cite{MMS1}, where the  use of Strichartz estimates
restricts  consideration to $p\ge 2$.  In this sense we go beyond the results for more regular potentials considered by
Mizumachi   \cite{mizumachi},  in turn related to \cite{SW1,SW2,PW,GNT}. In some of these older papers there is a clear interest at obtaining  the largest possible set of
of values for the exponent $p$.  However     they are severely restricted  by their  dependence      on   dispersive and/or Strichartz estimates,   not always sufficiently  robust in nonlinear settings, see \cite{strauss81}. The commutator method can be   more robust, as results such as   \cite{KMM2,KMM3} show. In the literature some  partial stability results have been obtained for subcritical nonlinearities in 2--D and 3--D by Kirr \textit{et  al.}
\cite{kirr1,kirr2} using dispersive estimates. But exactly like for the dispersion results on small solutions of cubic NLS's with potentials in \cite{Delort,Naumkin,Germain,MMS2}  or for the theory  initiated by Deift and Zhou \cite{DZ} on the Scattering Transform in non--Integrable Systems,
 the ultimate test will be how pliable,   widely   utilizable and not too technically complicated  will they be.  Our point here is that the theory in  \cite{KMM2,KMM3} seems the most promising.

Just as a final remark,  to emphasize once more the effectiveness of the commutator method for equations
like \eqref{1}, in Sect.\  \ref{sec:thm:dispersion}
 we will also give a very simple proof of  the following result
  for defocusing  equations \eqref{1} with non--trapping $\delta$  potential,
   which to our knowledge is not in the literature.

\begin{theorem}\label{thm:dispersion}  Consider equation \eqref{1}  with $q<0$,  $g \ge 0$ everywhere and
  $sg(s)- G(s) \ge 0$ for any $s\ge 0$ for $G$ defined in \eqref{eq:G}.  Then for any $\gamma >0$  there exists a  $C_ {\gamma  } >0$ s.t. for any $u_0\in H ^1(\R , \C )$ the corresponding strong solution $u(t)$ satisfies
\begin{equation}\label{eq:dispers-}
\begin{aligned}&     \int _0 ^\infty     \| u\|_{H^1_{-\gamma}}^2 dt < C_ {\gamma  }  \(       ({E (u_0) Q(u_0)} ) ^{\frac{1}{2}}+  Q(u_0) \) .
\end{aligned}
\end{equation}

\end{theorem}

\section{Proof of Theorem  \ref{thm:small en} } \label{sec:thm:small en}

\subsection{Notation and coordinates} \label{section:set up}

We have the following  ansatz, which is an elementary consequence of the Implicit Function Theorem.
\begin{lemma}\label{lem:decomposition}
There exist $c _0 >0$ and $C>0$ s.t.\ for all $u \in H^1$   with $\|u\|_{H^1}<c  _0 $, there exists a unique  pair $(z,\xi )\in   \C \times  P_c H^1 $
s.t.
\begin{equation}\label{eq:decomposition1}
\begin{aligned} &
u= Q[z]+\xi  \text{ with }  |z |+\|\xi \|_{H^1}\le C \|u\|_{H^1} .
\end{aligned}
\end{equation}
 The map  $u \to (z,\xi )$  is in $C^1 (D _{H^1}(0, c _0 ),
 \C   \times  H^1 )$.
\end{lemma}

\begin{proof}
Set
\begin{align*}
F(z,u):=\begin{pmatrix} \<u-Q[z],\varphi\>\\ \<u-Q[z],\im \varphi\> \end{pmatrix}.
\end{align*}
Then, by Proposition \ref{prop:bddst}, we see that $F\in C^1(D_{\C}(0,a_0)\times H^1;\R^2)$ and moreover
\begin{align}\label{eq:decomposition2}
\left.\frac{\partial F}{\partial(z_R,z_I)}\right|_{(z,u)=(0,0)}=-\begin{pmatrix}1 & 0\\ 0 & 1 \end{pmatrix}.
\end{align}
Therefore, by implicit function theorem, we have the conclusion.
\end{proof}

%
%

%
%
%
%

For $z,w\in \C$, we will use the notation
\begin{align*}
DQ[z]w:=\left.\frac{d}{d \varepsilon}\right|_{\varepsilon=0}Q[z+\varepsilon w].
\end{align*}
Notice that by $Q[e^{\im \theta}z]=e^{\im \theta}Q[z]$, we have
\begin{align*}
\im Q[z]=\left.\frac{d}{d \varepsilon}\right|_{\varepsilon=0}e^{\im \varepsilon}Q[ z]=\left.\frac{d}{d \varepsilon}\right|_{\varepsilon=0}Q[e^{\im \varepsilon} z]=\left.\frac{d}{d \varepsilon}\right|_{\varepsilon=0}Q[ z+\varepsilon \im z]=DQ[z]\im z.
\end{align*}
Further, for $w\in\C$, we set
\begin{align*}
f(w):=g(|w|^2)w.
\end{align*}

The well posedness of problem \eqref{1} in $H^1(\R)$  is considered by
 Goodman \textit{et al.}  \cite{GHW},
 Fukuizumi \textit{et al. } \cite{ozawa}  and \cite{adami}. The energy and mass conservation
 imply the global well posedness of our small solutions, with representation \eqref{eq:small en1} valid  for all times
along with the   bound \eqref{eq:small en2}.
So we can write the equation  \eqref{1} in terms of the ansatz  \eqref{eq:small en1}, or \eqref{eq:decomposition1}, and obtain the system
\begin{align}\label{11}
 & \< \im DQ(\dot z +\im E z), \im^{j-1} \varphi \>  =\< f(\xi)+\tilde f(z,\xi),\im^{j-1} \varphi \> \text{ for }j=1,2,\\&
\im \dot \xi =H_1\xi+f(\xi)+\tilde f(z,\xi) - \im DQ(\dot z +\im E z) \label{12},
\end{align}
where
\begin{align}\label{21}
\tilde f(z,\xi):=f(Q[z]+\xi)-f(\xi)-f(Q[z])
\end{align}

 In order to prove the last two points of Theorem \ref {thm:small en}, that is estimate \eqref{eq:small en2} and the limit \eqref{eq:small en22}, we will use
  the method  considered by  Kowalczyk \textit{et al.}  in \cite {KMM2} and  in  their very recent paper \cite{KMM3}.

  \subsection{The commutator method} \label{sec:commutator}

   Following  \cite{KMM3},
we introduce an even smooth  function $\chi :\R \to [-1,1]$ s.t.
\begin{align}\label{chi}
 &   \chi =1 \text{ in } [-1,1],  \   \chi =0 \text{ in } \R \setminus [-2,2], \chi '\le 0  \text{ in } \R_+. \end{align}
For $A\gg 1$  large enough which will be fixed later, we set
\begin{align}\label{zeta}
 &    \zeta _A (x):=\exp   \(  - \frac{|x|}{A}   \(1-\chi (x)\)   \)   \text{  and }  \psi _A(x) =\int _0^x  \zeta _A ^2(t)dt.
\end{align}
One can easily verify
\begin{align}\label{zetacomp}
e^{-\frac{|x|}{A}}\leq \zeta_A(x)\leq 2e^{-\frac{|x|}{A}}\text{ for }A\geq 4.
\end{align}
 To each function $\xi \in H^1 (\R )$ we can associate
\begin{align}\label{23}
 &  w:=\zeta  _A\xi . \end{align}
 Notice that  there exist  fixed constants $C$ and $A_0$ s.t.
 \begin{align}\label{231}
   \< (-\partial _x^2 +\delta ) w,w \>  \le C   \|  \xi \| ^2 _{H^1}  \text{  for all $A\ge A_0$.} \end{align}
Also, we have
\begin{align}\label{232}
\|w'\|_{L^2}^2+\|\<x\>^{-2}w\|_{L^2}^2\leq C \<(-\partial_x^2+ \delta)w,w\>,
\end{align}
where $\<x\>:=(1+|x|^2)^{1/2}$ and $C=12$.
Indeed, it suffices to bound the second term. Since
\begin{align*}
w(x)=w(0)+\int_0^x w'(s)\,ds,
\end{align*}
from H\"older inequality we have
\begin{align}\label{weightbound}
|w(x)|\leq |w(0)|+|x|^{1/2} \|w'\|_{L^2}\leq 2\<x\>^{1/2} \<(-\partial_x^2+ \delta)w,w\>^{\frac12}.
\end{align}
Thus, we have
\begin{align*}
\| \<x\>^{-2}w(x)\|_{L^2}^2 
\leq 4\int_\R \<x\>^{-3}\,dx \<(-\partial_x^2+ \delta)w,w\>\leq 12 \<(-\partial_x^2+ \delta)w,w\>.
\end{align*}
 We  consider  now  the
quadratic form
\begin{align}\label{Quadr}
 &   \mathcal{J}(\xi ) := 2^{-1} \< \im \xi ,  \( \frac{ \psi _A ' }{2}+ \psi _A \partial _x\)\xi \> .
\end{align}
By the well posedness of \eqref{1} we can consider
\begin{align*} \xi \in C^0 (\R , D(H_1)) \cap C^1(\R , L^2(\R, \C))  \subset  C^0 (\R , H^1(\R, \C)) \cap C^1(\R , L^2(\R, \C)) =:Y .
\end{align*}
We claim that  $ \mathcal{J}(\xi ) \in  C^1(\R , \R )$  with
\begin{align} \label{eq:Jxi}
\frac{d}{dt}\mathcal{J}(\xi )=\<\im \dot \xi, \( \frac{ \psi _A ' }{2}+ \psi _A \partial _x\) \xi\> .
\end{align}
Indeed we can consider  a sequence $\{  \xi  _n \} $ in  $  C^0 (\R , H^2(\R, \C)) \cap C^1(\R ,H^2(\R, \C))$ converging to $\xi $ in  $Y$ uniformly for $t$ on  compact sets.
The   functions  $ \mathcal{J}(\xi  _n)$ belong to $   C^1(\R , \R )$ and furthermore their derivatives satisfy \eqref{eq:Jxi} with $\xi $ replaced  by $\xi _n$.  From this formula
we derive that  the sequence $\{ \frac{d}{dt}\mathcal{J}(\xi _n ) \}$ converges uniformly on compact sets to the r.h.s. of  \eqref{eq:Jxi}. Since  $ \mathcal{J}(\xi  _n) \xrightarrow {n\to \infty }  \mathcal{J}(\xi   ) $
uniformly on compact sets,  we conclude that  $ \mathcal{J}(\xi ) \in  C^1(\R , \R )$  and that formula \eqref{eq:Jxi} is correct.

\noindent From  \eqref{eq:Jxi} we obtain
\begin{align}
  \frac{d}{dt}\mathcal{J}(\xi )  &+  \<\im DQ(\dot z +\im E z) , \( \frac{\psi _A ' }{2}+\psi _A \partial _x\) \xi   \>  \label{eq:virial2}  \\&
   =\<H_1\xi , \( \frac{ \psi _A ' }{2}+ \psi _A \partial _x\) \xi    \> + \< f(\xi),  \( \frac{ \psi _A ' }{2}+ \psi _A \partial _x\) \xi  \> +   \< \tilde f(z,\xi),  \( \frac{ \psi _A ' }{2}+ \psi _A\partial _x\) \xi   \> .\nonumber
\end{align}
The main result of this section is the following.

\begin{proposition} \label{prop:virial} There exist values $1\gg a _0>0$ and  $A\gg 1$   s.t. for $\xi\in P_c H^1$ and   $ |z| + \| \xi \| _{H^1}<a _0  $ we have
 \begin{equation}\label{prop:virial1}
\begin{aligned}
&
\text{r.h.s.\ of \eqref{eq:virial2}}
  \ge  \frac{1}{12}  \(\|w'\|_{L^2}+\|\<x\>^{-2}w\|_{L^2}^2\) \text{  for $w= \zeta _A \xi$.}
\end{aligned}
\end{equation}
\end{proposition}

The rest of this section is devoted to the proof of Proposition \ref{prop:virial}.

\noindent The key and single most important   term is the quadratic form singled out in the   following  lemma, see \cite{KMM2,KMM3}.

\begin{lemma}\label{lem:lemma1}  For $w=\zeta _A\xi $ we have the equality \begin{align}\label{Quadrmain1}
& \<  \( \frac{\psi ' _A}{2}+\psi _A \partial _x\) \xi , H_1\xi  \> =  \<   H _{\frac{1}{2}}w,w    \> + \frac{1}{2A}\< V  w   ,w   \> \text{ with }  V(x):= \chi ^{\prime\prime}(x) \ |x| + 2 \chi ^{\prime\prime}(x) \ \frac{x}{|x|}.
\end{align}

\end{lemma}

\begin{proof}
   Integrating by parts, see Corollary 8.10 \cite{brezis},  we obtain
\begin{align}
& \<  \( \frac{\psi _A' }{2}+\psi _A\partial _x\) \xi  , H_1\xi  \> =  - \frac{1}{2}\Re \int _\R \psi _A' \xi \overline{\xi}^{\prime\prime} dx - \frac{1}{2}  \int _\R\psi _A (|\xi ' | ^2)' dx\nonumber\\&
= \int _\R \psi _A' |\xi ' | ^2 dx + \frac{1}{4} \int _\R\psi _A  ^{\prime\prime} (|\xi ' | ^2)' dx + \frac{1}{2} \Re \(  \psi _A  ^{\prime }  (0)  \xi (0) \(  \overline{\xi} '(0^+)- \overline{\xi} '(0^-)  \)  \)
\nonumber \\& =
 \<  \psi '_A \xi ' ,\xi ' \> - \frac{1}{4}\<  \psi  ^{\prime\prime\prime}_A \xi   ,\xi  \>  -\frac{1}{2}\<   \delta \xi   ,\xi   \> , \label{eq:lem1:1}
\end{align}
where we used ${\xi} '(0^+)-  {\xi} '(0^-)=-\xi (0)$, $\psi _A (0) =\psi _A ^{\prime\prime }(0)=0 $ and $ \psi _A ^{\prime }(0)=1 $.

\noindent For the first term in the r.h.s. of \eqref{eq:lem1:1}, we have
\begin{align*}
  \<  \psi ' _A \xi ' ,\xi ' \>  &= \<\zeta_A^2\(\frac{w}{\zeta_A}\)', \(\frac{w}{\zeta_A}\)'\>= \< w' - \frac{\zeta _A'}{\zeta _A} w ,w' - \frac{\zeta _A'}{\zeta _A} w \> \\&  =\< w'   ,w'     \>  +\<  \( \frac{\zeta _A'}{\zeta _A}\)^2 w ,w  \> -2 \< w'   ,  \frac{\zeta '_A}{\zeta_A} w \> \\& = \< w'   ,w'     \>  +  \< \(  \( \frac{\zeta '_A}{\zeta _A}\) ' + \( \frac{\zeta '_A}{\zeta _A}\) ^2 \)     w   ,w      \>  =  \< w'   ,w'     \>  +  \<   \frac{\zeta ^{\prime\prime} _A}{\zeta _A}      w   ,w      \> ,
\end{align*}
and for the second term we have
\begin{align*}
&  - \frac{1}{4}\<  \psi  ^{\prime\prime\prime}_A \xi   ,\xi   \> =- \frac{1}{4}\<   \frac{(\zeta ^2_A) ^{\prime\prime}}{\zeta ^2_A} w   ,w   \>  = -\frac{1}{2}\<  \( \frac{ \zeta ^{\prime\prime}_A}{\zeta  _A}  + \( \frac{ \zeta ^{\prime } _A}{\zeta  _A} \) ^2  \) w   ,w   \> .
\end{align*}
Summing up we obtain
\begin{align}\label{eq:lem1:2}
\<  \( \frac{\psi _A' }{2}+\psi _A\partial _x\) \xi  , H_1\xi  \> =\< H_{\frac12}w,w\>+\frac12\< \(\frac{\zeta_A''}{\zeta_A}-\(\frac{\zeta_A'}{\zeta_A}\)^2\)w,w\>.
\end{align}
Finally,  from
\begin{align*}
\zeta_A'&=\frac{1}{A}\(\chi'(x)|x|+\(\chi(x)-1\)\frac{x}{|x|}\)\zeta_A \text{  and }\\
\zeta_A''&=\frac{1}{A^2}\(\chi'(x)|x|+\(\chi(x)-1\)\frac{x}{|x|}\)^2\zeta_A+\frac{1}{A}\(\chi''(x)|x|+2\chi'(x)\frac{x}{|x|}\)\zeta_A,
\end{align*}
we conclude
\begin{align}\label{eq:lem1:3}
A\(\frac{ \zeta ^{\prime\prime}_A}{\zeta  _A}  - \( \frac{ \zeta ^{\prime } _A}{\zeta  _A} \) ^2\)=\chi''(x)|x|+2\chi'(x)\frac{x}{|x|}=V(x).
\end{align}
Substituting \eqref{eq:lem1:3} into \eqref{eq:lem1:2}  we obtain \eqref{Quadrmain1}.
\end{proof}

The main step in the proof of Proposition \ref{prop:virial}  is the following lemma. \begin{lemma} \label{lemcor:errmain1}
  There exist a  fixed constant  $A_0>1$     s.t. for $A\ge A_0$ we have
  \begin{align} \label{lemerrmain1} &
   \<  \( \frac{\psi ' _A}{2}+\psi _A \partial _x\) \xi , H_1\xi  \>   \ge  \frac{1}{5}  \< (-\partial ^2_x+\delta )w,w\> \text{ for  $w= \zeta _A \xi$  and all $\xi \in \mathcal{H}_c[0]$.}
\end{align}
\end{lemma}
\proof  We use formula \eqref{Quadrmain1} singling out the the 1st term in the r.h.s. and writing it as
\begin{align*} &   \<  H_{\frac12} w,w    \>   =   \frac{1}{4}     \<   \(-\partial _x^2  + \delta \) w,w    \> + \frac{3}{4}  \<H_1w,w \>   .
\end{align*}
We now make the following claims.

\begin{claim} \label{lem:errorV}
  There exist  $A_0>0$ and $C_0>0$ s.t. for $A\ge A_0$ we have
  \begin{align} \label{eq:errorV1}
  \frac{1}{2A} \|  \<Vw,w \>  \|  &\le  \frac{C_0}{A} \< (-\partial ^2_x+\delta )w,w\> \text{ for all $w\in H^1$.}
\end{align}
\end{claim}

\begin{claim} \label{lem:errorOrth}
  There exist  $A_0>1$ and $C_0>0$ s.t. for $A\ge A_0$ we have
  \begin{align} \label{eq:errorOrth1}
    \<H_1w,w \>   &\ge  - \frac{C_0}{A^2 } \< (-\partial ^2_x+\delta )w,w\> \text{ for all $w= \zeta _A \xi$  with $\xi \in L^2_c$.}
\end{align}
\end{claim}

Let us assume Claims \ref{lem:errorV} and \ref{lem:errorOrth}. Then we conclude
\begin{align*}    \<  \( \frac{\psi ' _A}{2}+\psi _A \partial _x\) \xi , H_1\xi  \> &=  \<   H _{\frac{1}{2}}w,w    \> + \frac{1}{2A}\< V  w   ,w   \>   \\&=  \frac{1}{4}          \<   \(-\partial _x^2  + \delta \) w,w    \> + \frac{3}{4}  \<H_1w,w \> + \frac{1}{2A}\< V  w   ,w   \>    \\& \ge
 \(    \frac{1}{4}  -C_0\(\frac34 A^{-2} +A^{-1} \) \)      \<   \(-\partial _x^2  + \delta \) w,w    \>
\end{align*}
which yields immediately \eqref{lemerrmain1}. This, up to the proof of Claims  \ref{lem:errorV} and \ref{lem:errorOrth}, completes the proof of Lemma \ref{lemcor:errmain1}.  \qed

\textit{Proof of Claim  \ref{lem:errorV}}.
  By integrating
 \eqref{weightbound},
\begin{align*}
  \frac{1}{2A} \|  \<Vw,w \>  \|  &\le  \frac{2}{A}    \int _{\R}       |V(x)|   \<x\> dx        \< (-\partial ^2_x+\delta )w,w\> \text{ for all $w\in H^1$.}
\end{align*}
\qed

\textit{Proof of Claim  \ref{lem:errorOrth}}.
Since $w=\(\int_\R w\varphi_1\,dx \)\varphi_1 +P_c w$ and $\<H_1 P_c w,P_c w\>\geq 0$, we have
\begin{align*}
    \<H_1w,w \> = -\frac14 \left |  \int _\R w \varphi _1 dx  \right |   ^2   +\<H_1P_cw,P_cw\>  &\ge  - \frac{1}{4}   \left |  \int _\R w \varphi _1 dx  \right |   ^2    \text{ for all $w\in H^1 $.}
\end{align*}
On the other hand, for $w$ as in \eqref{eq:errorOrth1}  we have
\begin{align*}
\int _\R w \varphi _1 dx=\int_\R \xi \zeta_A \varphi_1\,dx=\int_\R \xi \(\zeta_A \varphi_1-\varphi_1\)\,dx=-\int_\R w \varphi_1\(\frac{1}{\zeta_A}-1\)\,dx.
\end{align*}
Since $e^{-\frac{|x|}{A}}\leq \zeta_A(x)\leq 1$ and $e^{|x|}-1\leq |x|e^{|x|}$ for all $x\in\R$, for $A\ge 4$
\begin{align*}
        \left |  \int _\R w \varphi _1 dx  \right |   &\le     \int _\R  |w |  \varphi _1  \(   \frac{1}{\zeta _A} -1\)    dx   \le   \frac{1}{\sqrt{2}}\int _\R  |w | e ^{-\frac{|x|}{2}}  \(  e ^{\frac{|x|}{A}} -1\)   dx\\&\le    \frac{1}{\sqrt{2} A}\int _\R  |w | e ^{-\frac{|x|}{2}   +  \frac{|x|}{A}   }  |x| dx   \le  \frac{1}{\sqrt{2} A}\int _\R  |w | e ^{-\frac{|x|}{4}   }  |x| dx.
\end{align*}
Furthermore, by \eqref{weightbound}
\begin{align*}
\frac{1}{\sqrt{2} A}\int _\R  |w | e ^{-\frac{|x|}{4}   }  |x| dx \le \frac{\sqrt 2}{A}\int_\R \<x\>^{3/2}e^{-\frac{|x|}{4}}\,dx \<(-\partial_x^2+\delta)w,w\>^{1/2}.
\end{align*}
This immediately leads to  the lower bound \eqref{eq:errorOrth1}.
\qed

%
%
By Lemma \ref{lemcor:errmain1} we have found a lower bound on the the 1st term in the r.h.s.\ of \eqref{eq:virial2}.
%
%
%
%
%
%
%

We now examine the contribution to  \eqref{prop:virial1} of the    term with $f(\xi)=g(|\xi|^2)\xi$.

\begin{lemma} \label{lem:errmain5}  For any  $\delta _0>0$  and   $A\gg 1$ there exists $a_0>0$    s.t. for $\| \xi \| _{H^1}\le a_0 $  we have
  \begin{align} \label{errmain51} &
 \left |  \<   \( \frac{ \psi _A ' }{2}+ \psi _A \partial _x\) \xi  ,f(\xi)   \>  \right |  \le \delta _0    \|w'\|_{L^2}^2 \text{ for  $w= \zeta _A \xi$ .}
\end{align}
\end{lemma}

\begin{proof}
We follow \cite{KMM2,KMM3}. Recall that $f(\xi)=g(|\xi|^2)\xi$. Consider the $G$ in  \eqref{eq:G}.
Then, we have
 \begin{align*}  &
   \<    \psi _A \partial _x  \xi  , g(|\xi | ^2) \xi   \>  =\Re\int_\R \psi_A g(|\xi|^2)\overline\xi \partial_x \xi\,dx=\frac{1}{2}\int_\R \psi_A \partial_xG(|\xi|^2)\,dx=-\frac{1}{2}\int_\R \psi_A' G(|\xi|^2)\,dx.
\end{align*}
Thus, by $\psi_A'=\zeta_A^2$, we have
\begin{align*}
\left|\<   \( \frac{ \psi _A ' }{2}+ \psi _A \partial _x\) \xi  , f(\xi)   \>\right|
\leq \int_\R \zeta_A^2 \(|g(|\xi|^{2})|\xi|^2|+2^{-1}|G(|\xi|^2)|\)\,dx\leq C  \int_\R \zeta_A^2 |\xi|^{2(p+1)}\,dx.
\end{align*}
Let $q=\frac{2p}{3}>0$.
Then, by the embedding $H^1(\R)\hookrightarrow L^\infty(\R)$, we have
 \begin{align*}
\int_\R \zeta_A^2|\xi|^{2(p+1)}\,dx  &\lesssim   \| \xi \| _{H^1}^{q} \int  _\R   \zeta  _A ^2     |\xi | ^ {2(p+1)-q} dx.
\end{align*}
Therefore, it suffices to prove
\begin{align}\label{err5-1}
\int_\R\zeta_A^2 |\xi|^{2(p+1)-q}\,dx\lesssim \|w'\|_{L^2}^2.
\end{align}
Following p. 793  \cite{KMM2}, since $4(p-q)+2=2(p+1)-q$,
\begin{align*}
 \int  _\R   \zeta  _A ^2     |\xi | ^ {2(p+1)-q} dx &=    \int  _\R   \zeta  _A ^{-(2p-q)}     |w | ^ {2(p+1)-q} dx \\&\lesssim \int  _0^\infty  e ^{\frac{2p-q}{A}x}    |w | ^ {2(p+1)-q} dx + \int  _{-\infty}^0  e ^{-\frac{2p-q}{A}x}    |w | ^ {2(p+1)-q} dx \\&
\leq -\frac{2A}{2p-q}   |w(0)  | ^ {2p-q}     +   \frac{A}{2p-q} \int  _\R e ^{\frac{2p-q}{A}|x|}   \left|( |w | ^ {2(p+1)-q})'\right|\,dx\\&\lesssim
 A\int_\R \zeta_A ^{-2p+q}  |w|^{2p-q+1} |w'|\,dx=A\int_\R \zeta_A |\xi|^{2p-q+1}|w'|\,dx\\&\leq
A\|\xi\|_{H^1}^q\int_\R \zeta_A  |\xi|^{2(p-q)+1} |w'|\,dx\\&
\leq A\|\xi\|_{H^1}^{q} \(\int_\R \zeta_A^2 |\xi|^{2(p+1)-q}\,dx\)^{\frac12} \|w'\|_{L^2}\\&\leq
\|w'\|_{L^2}^2+\frac14A^2\|\xi\|_{H^1}^{2q} \int_\R \zeta_A^2|\xi|^{2(p+1)-q}\,dx.
\end{align*}
Thus, taking $\frac14 A^2 a_0^{\frac{4}{3}p}\ll1$, we have \eqref{err5-1}.
%
\end{proof}


We now examine  the contribution to  \eqref{prop:virial1} of the  term with $\tilde f(z,\xi)$.
\begin{lemma} \label{lem:lipsc0}    There exist  $C_0>0$  and $r>0$  and a neighborhood $\mathcal{U}$    of the origin in $\C$ s.t.    for any pair $z,\xi \in \mathcal{U}$  we have
  \begin{align}\label{eq:lipsc2} &
|\tilde f(z,\xi)|\le  C_0 |Q[z]| ^r |\xi|  .
\end{align}
\end{lemma}
\proof  It is enough to set $ \zeta = Q[z]$ and then to prove
 \begin{align}\label{eq:lipsc21} &
|f(\zeta+\xi)-f(\xi)-f(\zeta)|\le  C_0 |\zeta| ^r |\xi|  .
\end{align}
We will prove \eqref{eq:lipsc21}  with $r=1$ if $2p\ge 1$ and with $r=2p$ if $2p\le 1$. With these values of $r$,
then  $|\zeta |\ge |\xi| $  implies  $| \zeta | ^r |\xi|\le | \xi | ^r |\zeta |$.  This means that it is enough to consider the case $|\zeta |\ge |\xi | $.

\noindent    If $|\zeta |\le 2 |\xi| $ we have $|\zeta |\sim    |\xi| $  and it is elementary to conclude that each of the
3 terms in the l.h.s. of  \eqref{eq:lipsc21} is $\lesssim | \zeta | ^r |\xi|$. Hence we are left with case  $|\zeta |\ge 2 |\xi| $.
Notice that $ \left | f(\xi)  \right |\lesssim    |\xi|   ^{2p+1}  \le      | \zeta | ^r |\xi|$. So it is enough to prove that  for a fixed $C_r>0$ we have
 \begin{align}\label{eq:lipsc3} &
 \left |  f(\zeta+\xi)-f(\zeta)  \right |   \le C_r | \zeta | ^r |w|.
\end{align}
  By \eqref{nonlinearity} we obtain the following, that implies \eqref{eq:lipsc3} and  completes the proof  of the lemma:
\begin{align*} &\left |  f(\zeta+\xi)-f(\zeta)  \right |    \le \int _{0}^{1} \left | \frac{d}{dt}  f  (\zeta +t \xi)  \right | dt \\& =   \int _{0}^{1}  \left |  g  (|\zeta +t \xi |^2)    \xi + 2 g'  (|\zeta +t \xi |^2) (\zeta +t w) \( \Re \(  \zeta \overline{w} \) +t|w|^2\) \right | dt \\& \le   C    |\zeta | ^{2p }  |\xi|  + C \(  |\zeta | ^{2p } |\xi|+ |\zeta | ^{2p -1 }  |\zeta | ^{2}\) \le   3 C       |\zeta | ^{2p}  |\xi|  .
\end{align*}   \qed

\begin{lemma} \label{lem:errmain7}   For any  $\delta _0>0$  and   $A\gg 1$ there exists $a _0>0$    s.t. for $\| (z,\xi) \| _{\C \times H^1}\le a_0 $  we have
  \begin{align*} &
 \left |  \<   \( \frac{ \psi _A ' }{2}+ \psi _A \partial _x\) \xi  , \tilde f(z,\xi) )   \>  \right |  \le \delta _0    \(\|w'\|_{L^2}^2+\|\<x\>^{-2}w\|_{L^2}^2\) \text{ for  $w= \zeta _A \xi$ .}
\end{align*}
\end{lemma}

\begin{proof}
First, we have
\begin{align*}
\left |  \<  \frac{ \psi _A ' }{2} \xi   ,\tilde f(z,\xi)   \>  \right |  \lesssim \int_\R |Q[z]||w|^2\,dx\lesssim a_0\|\<x\>^{-2}w\|_{L^2}^2.
\end{align*}
Next, since $|\partial_x \xi|\lesssim e^{2\frac{|x|}{A}}\(|w'|+|w|\)$, we have
\begin{align*}
\left |  \< \psi _A \partial_x \xi   ,\tilde f(z,\xi)   \>  \right |\lesssim A \int_\R |Q[z]|e^{\frac{3}{A}|x|}\(|w'|+|w|\)|w|\,dx\lesssim a_0A \(\|w'\|_{L^2}^2+\|\<x\>^{-2}w\|_{L^2}^2\).
\end{align*}
Therefore, taking $a_0A\ll1$, we have the conclusion.
\end{proof}

%
%
%
%

  \subsection{Closure of the estimates and completion of  the proof of  Theorem  \ref{thm:small en}} \label{sec:estimates}

We set
\begin{equation}\label{eq:Xnorm}
    \|w\|_X^2:=\|w'\|_{L^2}^2+\|\<x\>^{-2}w\|_{L^2}^2.
\end{equation}
In view of \eqref{eq:virial2}--\eqref{prop:virial1}, for any $T>0$, we have
\begin{align}   \label{sec:estimates1}&
\int_0^T\| w(t) \| _{X } ^2\,dt \lesssim   \epsilon ^2 + \(\int_0^T\| w \| _{X }^2\,dt\)^{1/2} \| \dot z +\im E z \| _{L^2 ((0,T) ) }.
\end{align}
From \eqref{11} we have
\begin{align*}
|\dot z + \im Ez |\lesssim \|\xi\|_{L^\infty}^{2p}\|w\|_{X} + |z| ^{\min(2p,1)}\|w\|_{X}.
\end{align*}
Thus,
\begin{align*}
\| \dot z +\im E z \| _{L^2 ((0,T) ) }  & \lesssim   \epsilon^{  {\min( 2p ,1) }} \(\int_0^T\| w(t) \| _X^2\,dt\)^{1/2}.
\end{align*}
Entering this in \eqref{sec:estimates1}  we obtain
\begin{align*}
 \| w \| _{L^2 ([0,T],X) }^2\lesssim  \epsilon ^2 +  \epsilon^{  {\min(2p ,1) }} \| w \| _{L^2 ([0,T],X) }^2 .
\end{align*}
  Since  $\min(2p ,1) >0$,    we   obtain $\| w \| _{L^2 ([0,T],X) }\le C_0  \epsilon$ for a    fixed $C_0>0$ and any $T$,  and so
\begin{align} \label{est1} &
\| w \| _{L^2 (\R _+,X) }   \le C_0  \epsilon,
\end{align}
which for $A$ sufficiently large implies   the estimate \eqref{eq:small en3} and ends the proof of Theorem \ref{thm:small en}.

\section{Proof of Theorem  \ref{thm:small en-} } \label{sec:thm:small en-}

We have the following  ansatz.
\begin{lemma}\label{lem:decomposition}
There exists $c _0 >0$ s.t.  there exists a $C>0$ s.t. for all $u \in H^1$   with $\|u\|_{H^1}<c  _0 $, there exists a unique  pair $(z,\eta )\in   \C \times  ( H^1 \cap \mathcal{H}_{c}[z])$
s.t.
\begin{equation}\label{eq:decomposition1}
\begin{aligned} &
u= Q[z]+\eta  \text{ with }  |z |+\|\eta \|_{H^1}\le C \|u\|_{H^1} .
\end{aligned}
\end{equation}
 The map  $u \to (z,\eta )$  is in $C^1 (D _{H^1}(0, c _0 ),
 \C   \times  H^1 )$.
\end{lemma}

\begin{proof}
   Set
\begin{align*}
F(z,u):=\begin{pmatrix} \<u-Q[z],\im D_1Q[z]\>\\ \<u-Q[z],\im D_2Q[z] \> \end{pmatrix}.
\end{align*}
Then, since here $2p>1$, we have $ F(z,u)\in C^1$ with formula \eqref {eq:decomposition2}      true for this function.  We conclude by Implicit Function Theorem.
\end{proof}

The following operator was introduced by Gustafson \textit{et al.} \cite{GNT}.

\begin{lemma} \label{lem:contcoo}
 There exists
$a _0>0$ such that for any $z\in\C$ with $|z|<a _0$  there exist $\alpha_j(z)$ ($j=1,2$) s.t. $\|e^{\gamma|x|}\alpha_j\|_{H^1}\lesssim |z|^{p-1}$ such that
the $\R$--linear operator $R[z]$ defined by
\begin{align}\label{contcoo-0}
R[z]\xi:=\xi +\<\im\xi,\alpha_1(z)\>\varphi +\<\im \xi,\alpha_2(z)\>\im \varphi ,
\end{align}
satisfies  $R[z]:\mathcal{H}[0]\to \mathcal{H}_c[z]$ and $ \left. P_c\right|_{\mathcal{H}_c[z]}=R[z]^{-1}$.

\end{lemma}

\begin{proof}
We search for $\beta_j=\beta_j(z,\xi)\in\R$ ($j=1,2$) which satisfy
\begin{align}\label{contcoo-1}
\<\im\(\xi+(\beta_1+\im\beta_2)\varphi \),D_jQ[z]\>=0\   \text{ for } j=1,2.
\end{align}
Since
\begin{align*}
D(z):=\<\varphi ,D_1Q[z]\>\<\im\varphi ,D_2Q[z]\>-\<\im \varphi D_1Q[z]\>\<\varphi ,D_2Q[z]\>=1+o(1),
\end{align*}
where $o(1)\to 0$ as $|z|\to 0$,
the system \eqref{contcoo-1} have a unique solution
\begin{align}\label{contcoo-2}
\begin{pmatrix}
\beta_1\\
\beta_2
\end{pmatrix}
=
D(z)^{-1}
\begin{pmatrix}
\<\varphi ,D_2Q[z]\> & -\<\varphi ,D_1Q[z]\>\\
\<\im\varphi ,D_2Q[z]\> & -\<\im\varphi ,D_1Q[z]\>
\end{pmatrix}
\begin{pmatrix}
\<\im \xi,D_1Q[z]\>\\
\<\im \xi, D_2 Q[z]\>
\end{pmatrix}.
\end{align}
Therefore, setting
\begin{equation}\label{contcoo-3}
\begin{aligned}
\alpha_1(z)&:= D(z)^{-1} \(\<\varphi ,D_2Q[z]\>D_1Q[z] -\<\varphi ,D_1Q[z]\>D_2Q[z]\),\\
\alpha_2(z)&:=D(z)^{-1} \(\<\im\varphi , D_2Q[z]\>D_1Q[z] -\<\im\varphi ,D_1Q[z]\>D_2Q[z]\),
\end{aligned}
\end{equation}
we see that $R[z]$ defined by \eqref{contcoo-0} with $\alpha_j$ given by \eqref{contcoo-3} satisfies $R[z]:\mathcal H[0]\to \mathcal H[z]$.
It is obvious that we have $P_c R[z]\xi=\xi$ for all $\xi\in \mathcal H[0]$.
Finally, we have
$
R[z]P_c\xi=\xi,
$
for all $\xi\in \mathcal H[z]$.
Indeed, since $R[z]P_c\xi$ is the unique element of $\mathcal H[z]$ of the form
$P_c\xi + \beta \varphi$,
we have the conclusion from the spectral decomposition $\xi = P_c\xi + \tilde \beta \varphi$ where $\tilde \beta=\int \xi\varphi\,dx$.
\end{proof}

Thanks to Lemmas \ref{lem:decomposition} and \ref{lem:contcoo} we conclude the following.
\begin{lemma} \label{lem:systcoo}  Picking $a_0>0$ small enough, we have
\begin{equation} \label{eq:systcoo1}
 u= Q[z]+R[z] \xi  \text{ for $(z,\xi )\in D_{\C }(0 , a_ 1) \times ( H^1\cap L^2_c) $}
\end{equation} such that    $(z,\xi )\to u$
 is $C^1$  and
\begin{equation} \label{eq:coo11}
  |z|+\| \xi \| _{H^1} \sim   \| u \| _{H^1}.
\end{equation}

\end{lemma} \qed

In terms decomposition \eqref{eq:small en1-}--\eqref{eq:small en2-}   equation  \eqref{1}
can be expressed as follows:
\begin{align}\label{11-}
 & \< \im DQ(\dot z +\im E z), D_jQ \> - \< \im\eta , D_jDQ(\dot z +\im E z) \> =\< f(\eta )+\tilde f(z,\eta ),D_jQ \> \text{ for }j=1,2\\& \im \dot \eta =H\eta+f(\eta )+\tilde f(z,\eta ) - \im DQ(\dot z +\im E z)  \label{12-},
\end{align}
where
\begin{align}\label{21-}
 & H\eta =H[z]\eta = H_1\eta +V[z]\eta   \text{ with } V[z]\eta =  - |Q[z]|^2 \eta - 2Q[z] \Re (\eta \overline{Q[z]} )\end{align}
and where $f(\eta )$ and  $\tilde f(z,\eta ) $ are defined like in  Sect.\ \ref{sec:thm:small en}. Like above, we consider
\begin{align}\label{Quadr-}
 &   \mathcal{J}(\eta ) = 2^{-1} \< \im \( \frac{ \psi _A ' }{2}+ \psi _A \partial _x\) \eta , \eta \> . \end{align}
Proceeding like in Sect. \ref{sec:commutator} we obtain
\begin{align}
  \frac{d}{dt}\mathcal{J}(\eta )  &-  \< \( \frac{\psi _A ' }{2}+\psi _A \partial _x\) \eta , \im DQ(\dot z +\im E z)  \>  \label{eq:virial2-}  \\&
   =\<  \( \frac{ \psi _A ' }{2}+ \psi _A \partial _x\) \eta , H_1\eta  \> + \<   \( \frac{ \psi _A ' }{2}+ \psi _A \partial _x\) \eta , V[z]\eta  \> \nonumber  \\&  +   \<   \( \frac{ \psi _A ' }{2}+ \psi _A\partial _x\) \eta , \widetilde{f}(z, \eta )   \>  +   \<   \( \frac{ \psi _A ' }{2}+ \psi _A\partial _x\) \eta , f(\eta )   \> .\nonumber
\end{align}

Then, like in Sect.  \ref{sec:commutator}, we have the following result.

\begin{proposition} \label{prop:virial-} There exist values $1\gg a _0>0$ and  $A\gg 1$   s.t. for $\xi\in P_c H^1$ and   $ |z| + \| \eta \| _{H^1}<a _0  $ we have
 \begin{equation}\label{prop:virial1-}
\begin{aligned}
&
\text{r.h.s.\ of \eqref{eq:virial2-}}
  \ge  \frac{1}{12}  \(\|w'\|_{L^2}+\|\<x\>^{-2}w\|_{L^2}^2\) \text{  for $w= \zeta _A \eta$.}
\end{aligned}
\end{equation}
\end{proposition}
\proof
The proof of Proposition \ref{prop:virial-}  is exactly like the proof of Proposition \ref{prop:virial}, except that in \eqref{eq:virial2-} there is an additional term (the 2nd in the 2nd line)
dealt with in Lemma \ref{lem:errmain3} below, and that   the analogue Lemma \eqref{lemerrmain1}  continues to be true with  an $\eta \in \mathcal{ H}[z]$ instead of  $\xi \in \mathcal{ H}[0]$.   We skip the elementary proof of the last point.    \qed

\begin{lemma} \label{lem:errmain3}
 There exist  $a_0>0$  and $C_0>0$  s.t. for $|z|\le a_0 $ we have
  \begin{align*}   &
 \left |  \<   \( \frac{ \psi _A ' }{2}+ \psi _A \partial _x\) \eta  , V[z]\eta  \>  \right |  \le C _0 |z|^2   \< (-\partial ^2_x+\delta )w,w\> \text{ for  $w= \zeta _A \eta$  and all $\eta \in H^1$.}
\end{align*}
\end{lemma}

\begin{proof}  In view of the fact that $ V[z]$  is symmetric for our inner product,  we have the following:
 \begin{align*}  & \<   \( \frac{ \psi _A ' }{2}+ \psi _A \partial _x\) \eta   , V[z] \eta   \> =  2 ^{-1} \<   \eta  ,     (|Q[z]|^2)'  \eta  +  2(Q[z] ) ' \Re (\eta \overline{Q[z]} ) +  2 Q[z]  \Re (\eta \overline{(Q[z] ) '} )   \> \\& \lesssim     |z| ^2\< (-\partial ^2_x+\delta )w ,w \> .
\end{align*}
\end{proof}

\textit{End of the proof of Theorem \ref{thm:small en-}.}
 From  \eqref{prop:virial1-} we have like in
\begin{align}   \label{sec:estimates1-}&
 \| w \| ^{2} _{L^2((0,t), X ) }\lesssim  2 \epsilon ^2 + \| w \|  _{L^2((0,t), X ) }\ \| \dot z +\im E z \| _{L^2 ((0,t) ) }.
\end{align}
From \eqref{12-}  we have
\begin{align*}
\| \dot z +\im E z \| _{L^2 ((0,t) ) }  & \lesssim \| \dot z +\im E z \| _{L^2 ((0,t) ) }     \|  \<  x\> ^{-2}w \| _{L^\infty ((0,t),L^2) }  \\& + \(  \|  z\|  _{L^\infty ((0,t) ) }   + \| \<  x\> ^{-2}w \| _{L^\infty  ((0,t),L^2) }  \)   \| w \|  _{L^2((0,t), X ) } ^2\\&  \lesssim  \epsilon  \| \dot z +\im E z \| _{L^2 ((0,t) ) } + \epsilon ^2 \| w \|  _{L^2((0,t), X ) },
\end{align*}
and hence
\begin{align*}
\| \dot z +\im E z \| _{L^2 ((0,t) ) }  & \lesssim  \epsilon ^2 \| w \| _{L^2 ((0,t),L^2) }.
\end{align*}
Entering this in \eqref{sec:estimates1-} we conclude, for a fixed $C_0>0$,
\begin{align*}  &
\| w \|  _{L^2((0,t), X ) }  \le C_0  \epsilon .
\end{align*}
We set now $\rho (t):= z (t) e^{ \im   \int _0^t E[z(s)]  ds}$. Then we have
\begin{align*} & \| \dot  \rho \| _{L^1 ((0,t) ) } =
\| \dot z +\im E z \| _{L^1 ((0,t) ) }   \\&  \lesssim  \| \dot  \rho \| _{L^1 ((0,t) ) }     \| \<  x\> ^{-2}w \| _{L^\infty  ((0,t),L^2) }+  \( \|  z\|  _{L^\infty ((0,t) ) }  +   \| \<  x\> ^{-2}w \| _{L^\infty  ((0,t),L^2) }\)   \| w \|  _{L^2((0,t), X ) } ^2.
\end{align*}
From this we derive
\begin{align*} & \| \dot  \rho \| _{L^1 (\R _+  ) } \lesssim  \epsilon    \| w \|  _{L^2(\R _+ , X ) } ^2\lesssim  \epsilon ^3 .
\end{align*}
The existence of $\rho _+$ and of the limit \eqref{eq:small en22}  follow. This ends the the proof of  \eqref{eq:small en3-}-- \eqref{eq:small en22}. \qed

\section{Proof of Theorem  \ref{thm:dispersion}   } \label{sec:thm:dispersion}

We know that there exists a unique global strong solution $u\in C ^{0}( \R , H^1(\R ,\C ))$, and  furthermore that energy and mass are constant
\begin{align*} &  E(u(t))=  \frac{1}{2} \| \partial _x u(t) \| _{L^2(\R )}^2  +
\frac{|q|}{2}  |   u(t,0)  |  ^2   +\frac{ 1}{2} \int _\R G(| u(t) |^2) dx
   = E(u_0), \\&  Q(u(t))= \frac{1}{2} \| u(t) \| _{L^2(\R )}^2 =  Q(u_0) .
\end{align*}
By well posedness and a density argument, it is enough to focus on the case $u_0\in D(H_q)$, so that \begin{align*} u \in C^0 (\R , D(H_1)) \cap C^1(\R , L^2(\R, \C))   .
\end{align*}
Then we consider $\mathcal{J}(u)$, defined like in \eqref{Quadr}, and by the same argument of
Sect. \ref{sec:commutator}  we have
\begin{align*} & \frac{d}{dt}\mathcal{J}(u )=\<\im \dot u, \( \frac{ \psi _A ' }{2}+
\psi _A \partial _x\) u\> =
\<\im \dot u, \( \frac{ \psi _A ' }{2}+ \psi _A \partial _x\) u\>  \\& =
\<  (-\partial ^2_x  +|q|\delta (x) ) u  + g(|u|^{2 })  u , \( \frac{ \psi _A ' }{2}+ \psi _A \partial _x\) u\> .
\end{align*}
By  computations similar to Lemma \ref{lem:lemma1},  for  $w=\zeta _A\xi$
 and for the $V(x)$ in \eqref{Quadrmain1}, we have
 \begin{align*} &   \<  (-\partial ^2_x  +|q|\delta (x) ) u
   , \( \frac{ \psi _A ' }{2}+ \psi _A \partial _x\) u\>    =
    \< \(-\partial ^2_x  +\frac{|q|}{2}\delta (x) \) w   , w \>   +
      \frac{1}{2A}\< V  w   ,w   \>  \\&
      \ge  \frac{1}{2} \< \(-\partial ^2_x  +\frac{|q|}{2}\delta (x) \) w   , w \>
\end{align*}
for $A\ge A_0$ with $A_0$ a fixed sufficiently large constant.

\noindent On the other hand, by   $\psi _A '>0$  and the argument in   the first few lines of Lemma \ref{lem:errmain5},
\begin{align*} &    \<   g(|u|^{2 })  u    , \frac{ \psi _A ' }{2} u \>
 +     \< g(|u|^{2 })  u          , \psi _A \partial _x  u\> =\frac{1}{2}
    \<   g(|u|^{2 })  |u |^2- G(|u|^{2 })      ,  \psi _A '   \> \ge 0.
\end{align*}
Hence, for fixed  constants
\begin{align*} &
\int_0^T\| w(t) \| _{X } ^2\,dt \lesssim  \mathcal{J}(u(T)) - \mathcal{J}(u_0) \lesssim \sqrt{E (u_0) Q(u_0)} +  Q(u_0),
\end{align*}
which yields Theorem \ref{thm:dispersion}.

\qed

\appendix
\section{Appendix.}\label{sec:implicit}
We prove Proposition \ref{prop:bddst}.

\begin{lemma}\label{lem:R}
Set $R:=\(\left.\(H_1+\frac14\)\right|_{P_c L^2}\)^{-1}$.
Then, for sufficiently small $\gamma> 0$, $R$ is a bounded operator from $L^2_\gamma$ to $H^1_\gamma$.
\end{lemma}

\begin{proof}
For  case $\gamma=0$  see Lemma 2.12 of \cite{MMS1}.
For the case $\gamma>0$, set $\chi_A(x):=\chi(x/A)$ where $\chi$ is given in \eqref{chi}.
Set $\mu_{\gamma,A}(x):=e^{\gamma   \sqrt{1+|x|^2} }\chi_A(x)$.
Then, multiplying $H_1 Ru =u$ by   $\mu_A$, we obtain
\begin{align*}
H_1 \mu_{\gamma,A}Ru =[H_1,\mu_{\gamma,A}]Ru + \mu_{\gamma,A}u.
\end{align*}
Notice that there exists a $C>0$ s.t.
\begin{align*}
\|[H_1,\mu_{\gamma,A}]u\|\leq C\ \gamma  \ \|\mu_{\gamma,A}u\|_{H^1} \text{ for all $\gamma\in[0,1]$ and $A\in[1,\infty)$.}
\end{align*}
This implies that for sufficiently small $\gamma>0$,
\begin{align*}
\|\mu_{\gamma,A}Ru\|_{H^1}\lesssim\|\mu_{\gamma,A}u\|_{L^2}\lesssim \|u\|_{L^2_\gamma}.
\end{align*}
Thus, taking $A\to \infty$, we have $\|Ru\|_{H^1_\gamma}\lesssim \|u\|_{L^2_\gamma}$.
\end{proof}

We consider  $\tilde h:[0,1]\times \R \to \R$ defined by
\begin{align}\label{def:tilh}
\tilde h(\rho, \mu):=g(\rho \mu^2)\mu.
\end{align}
For $\gamma<\frac12$, we set $h:[0,1]\times H^1_\gamma(\R,\R) \to L^2_\gamma(\R,\R)$ by
\begin{align}\label{def:h}
h(\rho,q) (x):=\tilde h(\rho,q(x))=g(\rho q(x)^2)q(x).
\end{align}
Notice that $q$ in \eqref{def:tilh} is a number but $q$ in \eqref{def:h} is a function.

\begin{lemma}\label{lem:esttilh}
We have $\tilde h\in C([0,1]\times \R,\R)\cap C^1((0,1]\times \R,\R)$ and the estimates
\begin{align}\label{est:tilh}
|\tilde h(\rho,\mu)|\lesssim \rho^p |\mu|^{2p+1}
\end{align}
and
\begin{align}\label{est:diftilh}
|\partial_\rho \tilde h(\rho,\mu)|\lesssim \rho^{p-1} |\mu|^{2p+1},\quad
|\partial_\mu \tilde h(\rho,\mu)|\lesssim \rho^p |\mu|^{2p}.
\end{align}
Furthermore, for $\rho\mu\neq 0$, $h$ is three times differentiable and we have
\begin{equation}\label{est:twicedifftildh}
\begin{aligned}
|\partial_\rho^2 \tilde h(\rho,\mu)|\lesssim \rho^{p-2}|\mu|^{2p+1},\quad
|\partial_\rho \partial_\mu \tilde h(\rho,\mu)|\lesssim \rho^{p-1}|\mu|^{2p},\quad
|\partial_\mu^2 \tilde h(\rho,\mu)|\lesssim \rho^{p}|\mu|^{2p-1}
\end{aligned}
\end{equation}
and
\begin{equation}\label{est:thirddifftildh}
\begin{aligned}
|\partial_\rho^3 \tilde h(\rho,\mu)|\lesssim \rho^{p-3}|\mu|^{2p+1},\quad |\partial_\mu \partial_\rho^3 h(\rho,\mu)|\lesssim \rho^{p-1}\mu^{2p}.
\end{aligned}
\end{equation}
If $p>\frac12$, we have $\tilde h\in  C^2((0,1]\times \R,\R)$.
\end{lemma}

\begin{proof}
By the definition of $\tilde h$, we have $C([0,1]\times \R,\R)\cap C^3((0,1]\times (\R\setminus \{0\}),\R)$.
Also, \eqref{est:tilh} is immediate from \eqref{nonlinearity} and \eqref{def:tilh}.
At $(\rho,q)=(\rho,0)$ with $\rho>0$, $\tilde h$ is differentiable w.r.t.\ $\rho$ and $\mu$ having $\partial_\rho \tilde h=\partial_q \tilde h=0$.
One can see this easily from $\tilde h(\rho,0)=0$ and
\begin{align*}
\tilde h(\rho+\epsilon,0)=0,\ |\tilde h(\rho,\epsilon)|\lesssim \epsilon^{2p+1}.
\end{align*}
Further, since for $\mu \neq 0$,
\begin{align}\label{eq:honedr}
\partial_\rho \tilde h(\rho,q)=g'(\rho\mu^2)\mu^3,\quad
\partial_\mu \tilde h(\rho,\mu)= g(\rho\mu^2)+2\rho g'(\rho\mu^2)\mu^2,
\end{align}
we have \eqref{est:diftilh} from \eqref{nonlinearity}, which imply that $\partial_\rho \tilde h$ and $\partial_\mu \tilde h$ are continuous at $(\rho,0)$.

Differentiating \eqref{eq:honedr} for $\rho,\mu\neq 0$, we have
\begin{align*}
\partial_\rho^2 \tilde h(\rho,q)&=g''(\rho\mu^2)\mu^5, \quad
\partial_\rho \partial_\mu \tilde h(\rho,\mu)=3g'(\rho\mu^2)\mu^2+2\rho g''(\rho\mu^2)\mu^4,\\
\partial_\mu^2 \tilde h(\rho,\mu)&=6\rho g'(\rho\mu^2)\mu+4\rho^2 g''(\rho\mu^2)\mu^3.
\end{align*}
and
\begin{align*}
\partial_\rho^3 \tilde h(\rho,q)=g'''(\rho\mu^2)\mu^7,\quad \partial_\mu \partial_\rho^2 \tilde h(\rho,q)=2 \rho g'''(\rho \mu^2) \mu^6.
\end{align*}
This implies that $\rho,\mu\neq 0$, we have \eqref{est:twicedifftildh} and \eqref{est:thirddifftildh}.
By \eqref{est:twicedifftildh}, for the case $p>1/2$, we see that $h$ is twice continuously differentiable at $(\rho,0)$ ($\rho\neq 0$) with
$$\partial_\rho^2 \tilde h(\rho,0)=\partial_\rho\partial_\mu \tilde h(\rho,0)=\partial_\mu^2 \tilde h(\rho,0)=0.$$
Therefore, we have the conclusion.
\end{proof}

\begin{lemma}\label{lem:diffh}
Let $\gamma\geq 0$.
Let $\tilde h$, $h$ be the functions given in \eqref{def:tilh} and \eqref{def:h}.
Then,
\begin{align}\label{hisC1}
h\in C([0,1]\times H^1_\gamma,L^2_\gamma)\cap C^1((0,1]\times H^1_\gamma,L^2_\gamma)
\end{align}
and
\begin{align}\label{diff:form}
\partial_\rho h(\rho,q)(x)=\partial_\rho \tilde h(\rho,q(x)),\  \(\partial_q h(\rho,q)v\)(x)=\partial_\mu \tilde h(\rho,q(x))v(x).
\end{align}
\end{lemma}

\begin{proof}
First of all, for $q\in H^1_\gamma$, we have $h(\rho,q)\in L^2_\gamma$.
Indeed, from \eqref{est:tilh},
\begin{align}\label{est:l2h}
\|h(\rho,q)\|_{L^2_\gamma}=\|\tilde h(\rho,q(\cdot))\|_{L^2_\gamma}\lesssim \rho^p \|q\|_{L^\infty}^{2p}\|q\|_{L^2_\gamma}\lesssim \rho^p \|q\|_{H^1_\gamma}^{2p+1},
\end{align}
which implies also that $h$ is continuous  at $\{0\}\times H^1$.

Next, we show \eqref{diff:form}.
For $(\rho,q)\in (0,1]\times H^1_\gamma$, and $|\epsilon|<\rho$,
\begin{align*}
&|\epsilon|^{-1}\| h(\rho+\epsilon,q)-h(\rho,q)-\epsilon \partial_\rho \tilde h(\rho,q)\|_{L^2_\gamma}=|\epsilon|^{-1}\| \tilde h(\rho+\epsilon,q)-\tilde h(\rho,q)-\epsilon \partial_\rho \tilde h(\rho,q)\|_{L^2_\gamma}\\&
=\int_0^1\| \partial_\rho \tilde h(\rho+\tau_1\epsilon,q)\,d\tau- \partial_\rho \tilde h(\rho,q)\|_{L^2_\gamma}\,d\tau_1\\&
\lesssim \int_0^1\int_0^1\tau_1|\epsilon|\| \partial_\rho^2 \tilde h(\rho+\tau_1\tau_2\epsilon,q)\|_{L^2_\gamma(\{x\in\R|q(x)\neq 0\})}\,d\tau_1d\tau_2\\&\lesssim |\epsilon| \int_0^1\int_0^1 \tau_1(\rho+\tau_1\tau_2 \epsilon)^{p-2}\,d\tau_1d\tau_2 \|q\|_{H^1_\gamma}^{2p+1}\to 0 \text{ as }  \epsilon\to 0.
\end{align*}
Similarly,
\begin{align}\nonumber
&\|v\|_{H^1_\gamma}^{-1}\| h(\rho,q+v)-h(\rho,q)-\partial_\mu \tilde h(\rho,q)v \|_{L^2_\gamma}=\|v\|_{H^1_\gamma}^{-1}\| \int_0^1 \(\partial_\mu \tilde h(\rho,q+\tau v)-\partial_\mu \tilde h(\rho,q )\)v\|_{L^2_\gamma}\\&\leq
\sup_{\tau\in[0,1]} \| \partial_\mu \tilde h(\rho,q+\tau v)-\partial_\mu \tilde h(\rho,q )\|_{L^\infty}\to 0 \text{ as } \|v\|_{H^1_\gamma}\to 0.\label{eq:hC11}
\end{align}
Here  we have used the fact that $\partial_\mu \tilde h$ is uniformly continuous in $[\frac\rho2,1]\times [-\|q\|_{L^\infty}-1,\|q\|_{L^\infty}+1]$ and $\|v\|_{L^\infty}\to 0$ if $\|v\|_{H^1_\gamma}\to 0$.
By similar estimate, we see that $\partial_\rho h$ and $\partial_q h$ are continuous in $(0,1]\times H^1_\gamma$.
From this, we have \eqref{hisC1}.
\end{proof}

\begin{lemma}\label{lem:secondh}
Let $p>1/2$.
Let $\tilde h$, $h$ be the functions given in \eqref{def:tilh} and \eqref{def:h}.
Then,
\begin{align}\label{hisC1+}
h\in C([0,1]\times H^1_\gamma,L^2_\gamma)\cap C^2((0,1]\times H^1_\gamma,L^2_\gamma),
\end{align}
and
\begin{equation}\label{seconddiffh}
\begin{aligned}
&\partial_\rho^2 h(\rho,q)(x)=\partial_\rho^2 \tilde h(\rho,q(x)),\  \(\partial_\rho\partial_q h(\rho,q)v\)(x)=\partial_\rho\partial_\mu \tilde h(\rho,q(x))v(x),\\&
\partial_q^2 h(\rho,q)(v,w)(x)=\partial_\mu^2 \tilde h(\rho,q(x))v(x)w(x).
\end{aligned}
\end{equation}
\end{lemma}

\begin{proof}
Since the argument is similar to the proof of Lemma \ref{lem:diffh} we omit it.
\end{proof}

\begin{lemma}\label{lem:diffe}
Let $\gamma\in [0,\frac12)$ and set
\begin{align*}
\mathfrak e(\rho ,q):=\<h(\rho,\varphi+q),\varphi\>.
\end{align*}
Then, $\mathfrak e\in C^1((0,1)\times H^1_\gamma(\R,\R),\R)$.
Moreover, if $p>1/2$, we have $\mathfrak e\in C^2((0,1)\times H^1_\gamma(\R,\R),\R)$.
\end{lemma}

\begin{proof}
Set $F\in C^\infty(L^2_\gamma;\R)$ by $F(h):=\<h,\varphi\>$.
Since $\mathfrak e(\rho ,q)=F\circ h (\rho+\varphi ,q)$, we immediately have the conclusion from Lemmas \ref{lem:diffh} and \ref{lem:secondh}.
\end{proof}

\begin{lemma}\label{lem:defPhi}
Let $\gamma\in [0,\frac12)$.
Set
\begin{align}\label{def:Phi}
\Phi(\rho ,q):=\mathfrak e(\rho ,q) (\varphi+q) - h(\rho ,\varphi+q).
\end{align}
Then, $\Phi \in C^1 ((0,1]\times P_c H^1_\gamma(\R,\R); P_cL^2_\gamma(\R,\R))$.
Moreover, if $p>1/2$, we have $\Phi \in C^2 ((0,1]\times P_c H^1_\gamma(\R,\R); P_cL^2_\gamma(\R,\R))$
\end{lemma}

\begin{proof}
From Lemmas \ref{lem:diffh}, \ref{lem:secondh} and \ref{lem:diffe}, it suffices to show $\Phi(\rho ,q)\in P_c L^2_\gamma$.
However, from the definition of $\mathfrak e$ we obtain the following, which yields the conclusion:
\begin{align*}
\<\Phi(\rho ,q),\varphi\>=\mathfrak e(\rho,q) -\<h(\rho,\varphi+q),\varphi\>=0.
\end{align*}
\end{proof}

\begin{lemma}\label{lem:defq}
Take $\gamma_0\in (0,1/2)$ such that the conclusion of Lemma \ref{lem:R} holds.
Then there exists $\rho _0>0$ s.t.\ there exists a unique $q \in C^k( (0,\rho _0),H^1_\gamma)$ with $k=1$ and $k=2$ if $p>\frac12$ satisfying
\begin{align}\label{eq:qfix}
q(\rho )=R \Phi(\rho ,q(\rho )),
\end{align}
and
\begin{align}\label{eq:est:qfix}
\|q(\rho )\|_{H^1_\gamma}\lesssim \rho ^p,\ \|\partial_\rho q(\rho )\|_{H^1_\gamma}\lesssim \rho ^{p-1}\text{ and }|\mathfrak e(\rho,q(\rho))|\lesssim \rho ^p.
\end{align}
Moreover, if $p>1/2$ we have
\begin{align}\label{eq:est:qfix2}
\|\partial_\rho^2 q(\rho)\|_{H^1_\gamma}\lesssim \rho^{p-2}.
\end{align}
\end{lemma}

\begin{proof}
By Lemma \ref{lem:R} and Lemma \ref{lem:diffh}, we have
\begin{align*}
&\| R\Phi(\rho ,q_1)-R\Phi(\rho ,q_2)\|_{H^1_\gamma}\lesssim \| \Phi(\rho ,q_1)-\Phi(\rho ,q_2)\|_{L^2_\gamma}\lesssim \| h(\rho ,\varphi+q_1)-h(\rho ,\varphi+q_2)\|_{L^2_\gamma}\\&\leq \int_0^1 \|\partial_q h(\rho,\varphi+q_2+\tau(q_1-q_2))\|_{L^\infty}\,d\tau  \| q_1-q_2 \|_{L^2_\gamma}\lesssim \rho ^p \|q_1-q_2\|_{H^1_\gamma},
\end{align*}
for $q_1,q_2\in \overline{D_{H^1_\gamma}(0,1)}$.
Therefore, there exists $\rho _0>0$ s.t.\
\begin{align*}
\| R\Phi(\rho ,q_1)-R\Phi(\rho ,q_2)\|_{H^1_\gamma(\R,\R)}\leq \frac12 \|q_1-q_2\|_{H^1_\gamma},
\end{align*}
for all $\rho \in (0,\rho _0)$ and $q_1,q_2\in \overline{D_{H^1_\gamma(\R,\R)}(0,1)}$.
Thus, by contraction mapping principle, there exists a unique $q\in \overline{D_{H^1_\gamma(\R,\R)}(0,1)}$ satisfying \eqref{eq:qfix}.
We call $q(\rho)$   the fixed point of $R\Phi(\rho,\cdot)$ and set
\begin{align*}
F(\rho,q):=q-R\Phi(\rho,q).
\end{align*}
Since one can show $\left.\partial_q F\right|_{(\rho,q)=(\rho,q(q))}$ is invertible by using the estimate we have prepared, by the Implicit Function Theorem  and by Lemma \ref{lem:defPhi} we have $q \in C^k( (0,\rho _0),H^1_\gamma)$ with $k=1$ and $k=2$ if $p>\frac12$.

We now prove \eqref{eq:est:qfix}.
First, by the fact that $q(\rho)$ is the fixed point of $R\Phi(\rho,\cdot)$, Lemma \eqref{lem:R} and \eqref{est:tilh} with the definition of $h$, $\Phi$, we have
\begin{align*}
\|q(\rho)\|_{H^1_\gamma}=\|R\Phi(\rho)\|_{H^1_\gamma}
\lesssim \|\Phi(\rho)\|_{L^2_\gamma}\lesssim \|h(\rho,\varphi+q(\rho))\|_{L^2_\gamma}\lesssim \rho^p.
\end{align*}
Next, by the definition of $\mathfrak e$, we have
\begin{align*}
|\mathfrak e(\rho,q(\rho))|\leq \|h(\rho,\varphi+q(\rho))\|_{L^2}\lesssim \rho^p.
\end{align*}
Finally, since
\begin{align*}
\partial_\rho q = R\partial_q \Phi \partial_\rho q +R\partial_\rho \Phi,
\end{align*}
and by the above argument, for sufficietly small $\rho>0$, we have $\|(\mathrm{Id}-R\partial_q \Phi)^{-1}\|_{H^1_\gamma\to H^1_\gamma}\leq 2$, we have the 2nd estimate of \eqref{eq:est:qfix} by
\begin{align*}
\|\partial_\rho q\|_{H^1_\gamma}=\| (\mathrm{Id}-R\partial_q \Phi)^{-1} R \partial_q \Phi\|_{H^1_\gamma}\lesssim \|\partial_q \Phi \|_{L^2_\gamma}\lesssim \|\partial_q h\|_{L^2_\gamma}\lesssim \rho^{p-1}.
\end{align*}
The estimate \eqref{eq:est:qfix2} can be proved similarly.
\end{proof}

\begin{proof}[Proof of Proposition \ref{prop:bddst}]
Set $a_0=\rho_0^{\frac12}$ where $\rho_0>0$ is given in Lemma \ref{lem:defq}.
Set
\begin{align*}
Q[z]:=z(\varphi + q(|z|^2))\text{ and }E(|z|^2):=-\frac14 + \mathfrak e(|z|^2,q(|z|^2)),
\end{align*}
where $q \in C^1((0,a_0^2),H^1_\gamma(\R,\R))$ is given in Lemma \ref{lem:defq}.
Then, \eqref{Q:gauge} and \eqref{eq:sp} are immediate from the definition of $Q$, $q$ and $\mathfrak e$.
The first and third inequality of \eqref{prop:bddst1} follow  from \eqref{eq:est:qfix}.

By Lemma \ref{lem:defq}, we have $Q\in C(D_\C(0,a_0),H^1_\gamma)\cap C^k(D_\C(0,a_0)\setminus \{0\},H^1_\gamma)$ for $k=1$ and $k=2$ for $p>\frac12$.
However, since
\begin{align*}
\| D_jQ[z]-\im^{j-1}\varphi\|_{H^1_\gamma}=\| \im^{j-1}q(|z|^2) + 2 q'(|z|^2)zz_j\|_{H^1_\gamma}\lesssim |z|^{2p},\ j=1,2,
\end{align*}
we see that $Q[z]$ is also continuously differentiable at $z=0$.
Here, we have set $z=z_1+\im z_2$ for $z_1,z_2\in\R$.
Similarly, if $p>1/2$, we see that $Q[z]$ is twice continuously differentiable at the origin  and satisfying the estimate \eqref{prop:bddst4}.
This finishes the proof.
\end{proof}

\section*{Acknowledgments} S.C. was supported by a grant FRA 2018 from the University of Trieste.
M.M. was supported by the JSPS KAKENHI Grant Number 19K03579, JP17H02851 and JP17H02853..

Department of Mathematics and Geosciences,  University
of Trieste, via Valerio  12/1  Trieste, 34127  Italy.
{\it E-mail Address}: {\tt scuccagna@units.it}

Department of Mathematics and Informatics,
Faculty of Science,
Chiba University,
Chiba 263-8522, Japan.
{\it E-mail Address}: {\tt maeda@math.s.chiba-u.ac.jp}

\end{document}